\documentclass[12pt]{amsart}

\def\R{\mathbb{R}}
\def\calC{\mathcal{C}}

\def\bmu{\bar{\mu}}
\def\ba{\bar{a}}
\def\brho{\bar{\rho}}

\theoremstyle{plain}
	\newtheorem{theorem}{Theorem}[section]
	\newtheorem{lemma}[theorem]{Lemma}
	\newtheorem{corollary}[theorem]{Corollary}
	
	\newtheorem{proposition}[theorem]{Proposition}
	\newtheorem{remark}[theorem]{Remark}
	
\theoremstyle{plain}

\makeatletter
 \@addtoreset{equation}{section}
\makeatother

\begin{document}
\title[Exact eigenvalues and eigenfunctinos]{Exact eigenvalues and eigenfunctions for\\ a one-dimensional Gel'fand problem}
\author{Yasuhito Miyamoto}
\thanks{The first author was supported by JSPS KAKENHI Grant Number 16K05225.}
\address{Graduate School of Mathematical Sciences, The University of Tokyo, 3-8-1 Komaba, Meguro-ku, Tokyo 153-8914, Japan}
\email{miyamoto@ms.u-tokyo.ac.jp}

\author{Tohru Wakasa}
\address{Department of Basic Sciences, Kyushu Institute of Technology, Sensuicho, Tobata-ku, Kitakyushu, Fukuoka 804-8550,
Japan}
\email{wakasa@mns.kyutech.ac.jp}

\begin{abstract}
It is known that every positive solution of a one-dimensional Gel'fand problem can be written explicitly.
In this paper we give exact expressions of all the eigenvalues and eigenfunctions of the linearized eigenvalue problem at each solution.
We study asymptotic behaviors of eigenvalues and eigenfunctions as the $L^{\infty}$-norm of the solution goes to the infinity.
We also study the problem $u''+\lambda e^{-u}=0$ and the associated linearized problem.
\end{abstract}
\date{\today}
\maketitle

\section{Introduction and main results}
We consider the one-dimensional Gel'fand problem
\begin{equation}\label{GP+}
\begin{cases}
u''+\lambda e^u=0, & -1<x<1,\\
u(-1)=u(1)=0,
\end{cases}
\end{equation}
where $\lambda>0$ is a parameter.
A history of (\ref{GP+}) including multi-dimensional cases can be found in \cite[Section 1]{JS02}.
The set of the positive solutions of (\ref{GP+}) becomes a curve $\calC:=\{(\lambda,u)\}$ which can be parametrized by $\left\|u\right\|_{L^{\infty}([-1,1])}$.
Each solution on $\calC$ has the exact expression
\begin{equation}\label{S1E0-}
(\lambda,u)=\left(2e^{-\alpha}{\rm arcosh}^2\left(e^{{\alpha}/{2}}\right),\alpha-2\log\cosh\left(\sqrt{\frac{\lambda}{2}}e^{{\alpha }/{2}}x\right)\right),
\end{equation}
where $\alpha=\left\|u\right\|_{L^{\infty}([-1,1])}$ and $w={\rm arcosh}(z)$ denotes the inverse function of $z=\cosh(w):=(e^w+e^{-w})/2$.

Let $\tau(\alpha):={\rm arcosh}\left(e^{{\alpha}/{2}}\right)$.
Then,
\begin{equation}\label{tau+}
\alpha=2\log\cosh\tau,
\end{equation}
$\tau(0)=0$, $\tau(\infty)=\infty$, and $\tau$ is a homeomorphism from $[0,\infty)$ to $[0,\infty)$.
It is convenient to use $\tau$ ($0\le\tau<\infty$) as an independent variable instead of $\alpha$.
The above solution $(\lambda,u)$ can be written as follows:
\begin{equation}\label{S1E0}
(\lambda(\tau),u(x,\tau))=\left(\frac{2\tau^2}{\cosh^2\tau},2\log\frac{\cosh\tau}{\cosh(\tau x)}\right).
\end{equation}
Then, $\calC=\{(\lambda(\tau),u(x,\tau))|\ \tau>0\}$.
It is well known that the solution curve $\calC$ starts from $(\lambda,u)=(0,0)$, bends back once, and blows up at $\lambda=0$.
Since
\begin{equation}\label{S1E1+2}
\lambda'(\tau)=\frac{4\tau}{\cosh^2\tau}(1-\tau\tanh\tau),
\end{equation}
we see the following: $\lambda'(\tau)>0$ for $0<\tau<\tau_1$, $\lambda'(\tau_1)=0$, $\lambda'(\tau)<0$ for $\tau>\tau_1$, and $\lim_{\tau\to\infty}\lambda(\tau)=0$.
Here $\tau_1$ is the (unique) solution of the equation
\begin{equation}\label{S1E1+1}
\tau\tanh\tau=1\ \ \textrm{and}\ \ \tau>0.
\end{equation}

Eigenvalues and eigenfunctions are important in both qualitative and quantitative studies of a steady state of the associated parabolic problem.
In this paper we give exact expressions of the eigenvalues and eigenfunctions of the problem
\begin{equation}\label{EVP+}
\begin{cases}
\varphi''+\lambda(\tau) e^{u(\tau)}\varphi=-\mu\varphi, & -1<x<1,\\
\varphi(-1)=\varphi(1)=0.
\end{cases}
\end{equation}
In the last section we also consider the problem
\begin{equation}\label{GP-}
\begin{cases}
u''+\lambda e^{-u}=0,\ -1<x<1,\\
u(-1)=u(1)=0
\end{cases}
\end{equation}
and give exact expressions of the eigenvalues and eigenfunctions of the linearized problem
\begin{equation}\label{EVP-}
\begin{cases}
\varphi''+\lambda e^{-u}=-\mu\varphi, & -1<x<1,\\
\varphi(-1)=\varphi(1)=0.
\end{cases}
\end{equation}
All the results about (\ref{EVP-}) are summarized in Theorem~\ref{T3}.

The first main result is about the exact expression of the eigenvalue of (\ref{EVP+}).
\begin{theorem}\label{T1}
Let $\{\mu_j\}_{j=1}^{\infty}$, $\mu_1<\mu_2<\cdots$, be the eigenvalues of (\ref{EVP+}).
Let $\tau_1$ be the solution of (\ref{S1E1+1}).
Then the following hold:\\
(i) If $0<\tau<\tau_1$, then $\mu_j>0$ for $j\ge 1$ and $\mu_j$, $j\ge 1$, is a (unique) solution of the problem
\begin{equation}\label{T1E1}
\tan\left(\sqrt{\mu_j}-\frac{\pi}{2}(j-1)\right)=\frac{\sqrt{\mu_j}}{\tau\tanh\tau}\ \ \textrm{and}\ \ \frac{\pi}{2}(j-1)<\sqrt{\mu_j}<\frac{\pi}{2}j.
\end{equation}
(ii) If $\tau=\tau_1$, then $\mu_1=0$, $\mu_j>0$ for $j\ge 2$, and $\mu_j$, $j\ge 2$, satisfies (\ref{T1E1}).\\
(iii) If $\tau>\tau_1$, then $\mu_1<0$, $\mu_1$ is a (unique) negative solution of
\begin{equation}\label{T1E2}
\tanh\sqrt{-\mu_1}=\frac{\sqrt{-\mu_1}}{\tau\tanh\tau},
\end{equation}
$\mu_j>0$ for $j\ge 2$, and $\mu_j$, $j\ge 2$, satisfies (\ref{T1E1}).
\end{theorem}
It is known that the portion of the curve $\{(\lambda(\tau),u(x,\tau))|\ 0<\tau<\tau_1\}\subset\calC$ consists only of minimal solutions.
Hence, each solution is stable.
Indeed, it follows from Theorem~\ref{T1}~(i) that $\mu_j>0$ for $j\ge 1$.

The second main result is about the exact expression of the eigenfunction.
\begin{theorem}\label{T2}
Let $\{\mu_j\}_{j=1}^{\infty}$ be the eigenvalues given in Theorem~\ref{T1}, and let $\varphi_j(x)$ be the eigenfunction corresponding to $\mu_j$.
Let $\tau_1$ be the solution of (\ref{S1E1+1}).
Then the following hold:\\
(i) If $0<\tau<\tau_1$, then, for $j\ge 1$,\\
\begin{equation}\label{T2E1}
\varphi_j(x)=\sqrt{\frac{\mu_j}{\tau^2}+\tanh^2(\tau x)}\sin\left(\sqrt{\mu_j}x+\arctan\left(\frac{\tau\tanh(\tau x)}{\sqrt{\mu_j}}\right)+\frac{\pi}{2}j\right).
\end{equation}
(ii) If $\tau=\tau_1$, then
\begin{equation}\label{T2E2}
\varphi_1(x)=\tanh\tau_1-x\tanh(\tau_1 x),
\end{equation}
and $\varphi_j(x)$, $j\ge 2$, is given by (\ref{T2E1}).\\
(iii) If $\tau>\tau_1$, then
\begin{equation}\label{T2E3}
\varphi_1(x)=\sqrt{-\mu_1}\cosh(\sqrt{-\mu_1}x)-\tau\sinh(\sqrt{-\mu_1}x)\tanh(\tau x),
\end{equation}
and $\varphi_j(x)$, $j\ge 2$, is given by (\ref{T2E1}).
\end{theorem}
Using Theorems~\ref{T1} and \ref{T2}, we prove the following asymptotic behaviors:
\begin{corollary}\label{C1}
Let $\{\mu_j\}_{j=1}^{\infty}$ be the eigenvalues given in Theorem~\ref{T1}.
Then the following hold:\\
(i) For $j\ge 1$,\\
\begin{equation}\label{C1E1}
\lim_{\tau\downarrow 0}\sqrt{\mu_j}=\frac{\pi}{2}j.
\end{equation}
(ii)
\begin{equation}\label{C1E2}
\lim_{\tau\to\infty}\mu_1=-\infty\ \textrm{and, for}\ j\ge 2, \lim_{\tau\to\infty}\sqrt{\mu_j}=\frac{\pi}{2}(j-1).
\end{equation}
\end{corollary}
\begin{corollary}\label{C2}
Let $\varphi_j(x)$ be the eigenfunction given in Theorem~\ref{T2}.\\
(i) As $\tau\to\infty$,
\[
\frac{1}{\tau}\varphi_1\left(\frac{y}{\tau}\right)\to\frac{1}{\cosh y}\ \ \textrm{in}\ \ C_{loc}(\R).
\]
(ii) For $j\ge 2$, as $\tau\to\infty$,
\[
\varphi_j(x)\to\bar{\varphi}_j(x)
\ \ \textrm{in}\ \ C_{loc}([-1,0)\cup(0,1]),
\]
where
\[
\bar{\varphi}_j(x):=
\begin{cases}
\sin\left(\frac{\pi}{2}(j-1)x+\frac{\pi}{2}(j+1)\right) & \textrm{if}\ 0<x\le 1,\\
\sin\left(\frac{\pi}{2}(j-1)x+\frac{\pi}{2}(j-1)\right) & \textrm{if}\ -1\le x<0.
\end{cases}
\]
\end{corollary}
It follows from Corollary~\ref{C2}~(i) that the first eigenfunction $\varphi_1(x)$ converges to $\pi\delta_0(x)$ in a weak sense, where $\delta_0(x)$ is the delta measure.

\begin{remark}
Let $\Omega$ be a two-dimensional bounded domain with smooth boundary.
Let $\{(\lambda_i,u_i)\}_{i=1}^{\infty}$ be a sequence of solutions of the problem
\[
\begin{cases}
\Delta u+\lambda e^u=0, & \textrm{in}\ \Omega,\\
u=0, & \textrm{on}\ \partial\Omega.
\end{cases}
\]
Nagasaki-Suzuki~\cite{NS90} showed that as $\lambda_i\downarrow 0$, the sequence $\left\{\int_{\Omega}\lambda_ie^{u_i}dx\right\}_{i=1}^{\infty}$ accumulates at $0$, $8\pi m$ $(m\in\{1,2,3,\ldots\})$, or $\infty$.
In one-dimensional case by direct calculation we have
\[
\int_{-1}^1\lambda(\tau) e^{u(\tau)}dx=2\tau(\cos(2\arctan e^{-\tau})-\cos(2\arctan e^{\tau}))\to\infty\ \ \textrm{as}\ \ \tau\to\infty,\ \ \textrm{and}
\]
\[
\int_{-1}^1\lambda(\tau)e^{u(\tau)}dx\to 0\ \ \textrm{as}\ \ \tau\to 0.
\]
Therefore, $\left\{\int_{-1}^1\lambda(\tau_i)e^{u(\tau_i)}dx\right\}_{i=1}^{\infty}$ does not accumulate at a finite positive number as $\lambda(\tau_i)\downarrow 0$.
\end{remark}

\begin{remark}
By (\ref{S1E0}) we see that $\sqrt{\lambda(\tau)e^{u(\tau)}}=2\tau/\cosh(\tau x)$.
Since
\[
\int_{-1}^1\sqrt{\lambda(\tau)e^{u(\tau)}}dx=2\sqrt{2}(\arctan e^{\tau}-\arctan e^{-\tau})\to\sqrt{2}{\pi}\ \ \textrm{as}\ \ \tau\to\infty,
\]
we can show that, as $\tau\to\infty$, $\sqrt{\lambda(\tau)e^{u(\tau)}}$ converges to $\sqrt{2}\pi\delta_0(x)$ in a weak sense.
\end{remark}

Let us mention technical details.
Wakasa-Yotsutani~\cite{WY08} constructed a theory, which is explained in Section~2 of the present paper, in order to obtain exact eigenvalues and eigenfunctions for a one-dimensional elliptic problem with general nonlinearity.
They derived a key ODE (\ref{S1E5_5}) below, and show that if the ODE has an exact solution, then all the eigenvalues and eigenfunctions can be written explicitly.
Then they applied the theory to a Neumann problem in the case $f(u)=\sin u$ (resp. $f(u)=u-u^3$) and obtained all exact eigenvalues in \cite{WY08} (resp. \cite{WY15}) and all exact eigenfunctions in \cite{WY10} (resp. \cite{WY16}).
See also \cite{W06} for the case $f(u)=u-u^3$.
In this paper we obtain exact solutions (\ref{S1E8}) and (\ref{S3E1-2}) of the key ODE (\ref{S1E5_5}) and apply the theory to a Dirichlet problem in the case $f(u)=e^u$.
We do not use the Jacobi elliptic functions, while they appear in the case $f(u)=u-u^3$ or $f(u)=\sin u$.

This paper consists of five sections.
In Section~2 we recall basic properties about eigenvalues and eigenfunctions.
We also recall the theory of \cite{WY08} about exact eigenvalues and eigenfunctions.
In Section~3 we give an exact expression of the eigenfunction when the associated eigenvalue is positive.
We see that all the eigenvalues except the first one are positive.
Thus, only the first eigenvalue may be non-positive.
In Section~4 we study the case where the first eigenvalue is zero.
In Section~5 we study the case where the first eigenvalue is negative.
Moreover, we prove Theorems~\ref{T1} and \ref{T2} and Corollaries~\ref{C1} and \ref{C2}, using results obtained in Sections 3, 4, and 5.
In Section~6 we consider the problem (\ref{GP-}) and give exact expressions of the eigenvalues and eigenfunctions of (\ref{EVP-}).

\section{General nonlinearity}
Let $f$ be an arbitrary $C^2$-function.
We consider the one-dimensional elliptic Dirichlet problem with general nonlinearity
\begin{equation}\label{S1E2}
\begin{cases}
u''+\lambda f(u)=0, & -1<x<1,\\
u(-1)=u(1)=0.
\end{cases}
\end{equation}
Let $(\lambda,u(x))$ be a solution of (\ref{S1E2}).
The linearized eigenvalue problem at $(\lambda,u(x))$ becomes
\begin{equation}\label{S1E3}
\begin{cases}
\varphi''+\lambda f'(u)\varphi=-\mu\varphi, & -1<x<1,\\
\varphi(-1)=\varphi(1)=0.
\end{cases}
\end{equation}
First, we recall basic properties about the eigenvalue and eigenfunction for the one-dimensional problem (\ref{S1E3}).
\begin{proposition}\label{P1}
Let $\{\mu_j\}_{j=1}^{\infty}$ be the eigenvalues of (\ref{S1E3}), and let $\varphi_j(x)$ be the eigenfunction corresponding to $\mu_j$.
Then the following hold:\\
(i) Each eigenvalue $\mu_j$ is real and simple, and $\{\mu_j\}_{j=1}^{\infty}$ satisfies $\mu_1<\mu_2<\cdots$.\\
(ii) For $j\ge 1$, the zero number $\sharp\{x\in(-1,1)|\ \varphi_j(x)=0\}$ is $j-1$.\\
(iii) Assume that the solution $u(x)$ of (\ref{S1E2}) is even.
If $j(\ge 1)$ is odd, then $\varphi_j(x)$ is an even function.
If $j(\ge 2)$ is even, then $\varphi_j(x)$ is an odd function.
\end{proposition}
We omit the proof of Proposition~\ref{P1}.
See \cite[Theorem 2.1 in p.212]{CL55} for (i) and (ii).
The assertion (iii) follows from an easy symmetric argument.

We consider the positive solution of (\ref{S1E2}).
Then every solution is an even function, and hence the conclusions of Proposition~\ref{P1}~(iii) hold.

Next, we study exact expressions of the eigenvalue and eigenfunction.
We look for the eigenfunction of the form $\varphi(x)=\sqrt{\phi(x)}$.
Substituting $\sqrt{\phi(x)}$ into (\ref{S1E3}), we see that $\phi(x)$ satisfies
\begin{equation}\label{S1E4}
2\phi\phi''-\phi'^2+4(\lambda f'(u)+\mu)\phi^2=0.
\end{equation}
Let $\Phi(x):=2\phi''\phi-\phi'^2+4(\lambda f'(u)+\mu)\phi^2$.
Then,
\[
\Phi'(x)=
2\phi\left\{\phi'''+4(\lambda f'(u)+\mu)\phi'+2f''(u)u'\phi\right\}.
\]
We consider the equation
\begin{equation}\label{S1E4+1}
\phi'''+4(\lambda f'(u)+\mu)\phi'+2f''(u)u'\phi=0.
\end{equation}
If $\phi(x)$ is a solution of (\ref{S1E4+1}), then $\Phi'(x)=0$, and hence $\Phi(x)$ is constant.
In particular, $\Phi(x)=\Phi(0)$, i.e.,
\begin{multline}\label{S1E5}
2\phi''(x)\phi(x)-\phi'(x)^2+4(\lambda f'(u(x))+\mu)\phi(x)^2\\
=2\phi''(0)\phi(0)-\phi'(0)^2+4(\lambda f'(u(0))+\mu)\phi(0)^2.
\end{multline}
When $\Phi(0)=0$, then $\sqrt{\phi(x)}$ is a candidate of the eigenfunction, since (\ref{S1E4}) holds.
However, we show that a candidate of the eigenfunction can be constructed even if $\Phi(0)\neq 0$.

We study (\ref{S1E4+1}).
We look for a solution of the form $\phi(x)=h(u(x))$.
In other words we construct the eigenfunction, utilizing the shape of a general solution $u(x)$.
Substituting $h(u(x))$ into (\ref{S1E4+1}), we see that $h(u)$ satisfies the following key equation:
\begin{equation}\label{S1E5_5}
2(F(\alpha)-F(u))h'''-3f(u)h''+\left(3f'(u)+4\frac{\mu}{\lambda}\right)h'+2f''(u)h=0,
\end{equation}
where we use $\alpha:=u(0)$, $u'(0)=0$, $u'(x)^2=2\lambda(F(\alpha)-F(u(x)))$, and $u''=-\lambda f(u)$.
It is worth noting that the independent variable of (\ref{S1E5_5}) becomes $u$ and that the explicit $x$-dependence disappears.
Thus, (\ref{S1E5_5}) does not depend on each solution $u$ of (\ref{S1E2}).
One of the important results in \cite{WY08} is finding the equation (\ref{S1E5_5}).
Substituting $h(u(x))$ into (\ref{S1E5}), we have
\begin{equation}\label{S1E6}
(F(\alpha)-F(u))(2h''h-h'^2)-f(u)hh'+2\left(f(u)+\frac{\mu}{\lambda}\right)h^2=\rho,
\end{equation}
where
\begin{equation}\label{S1E6+1-}
\rho:=-f(\alpha)h(\alpha)h'(\alpha)+2\left(f'(\alpha)+\frac{\mu}{\lambda}\right)h(\alpha)^2.
\end{equation}

Hereafter, we construct an eigenfunction.
We assume that a solution $h(u)$ of (\ref{S1E5_5}) is obtained.
Then we look for the eigenfunction of the following form:
\begin{equation}\label{S1E6+1}
\varphi(x)=\sqrt{h(u(x))}W(\theta(x)).
\end{equation}
The functions $W(\theta)$ and $\theta(x)$ are defined later.
By (\ref{S1E6}) we have
\begin{multline}\label{S1E7}
\frac{d^2\sqrt{h(u(x))}}{dx^2}+\lambda f'(u)\sqrt{h(u(x))}+\mu \sqrt{h(u(x))}\\
=\frac{1}{4h\sqrt{h}}\left\{ 2\lambda(F(\alpha)-F(u))(2hh''-h'^2)-2\lambda f(u)hh'+4(\lambda f'(u)+\mu)h^2\right\}=\frac{2\lambda\rho}{4h\sqrt{h}}.
\end{multline}
Substituting (\ref{S1E6+1}) into (\ref{S1E3}), by (\ref{S1E7}) we have
\begin{align}
0&=
\frac{d^2\sqrt{h}}{dx^2}W+2\frac{d\sqrt{h}}{dx}W'\theta'+\sqrt{h}(W''\theta'^2+W'\theta'')+\lambda f'(u)\sqrt{h}W+\mu\sqrt{h}W\nonumber\\
&=\sqrt{h}\theta'^2\left( W''+\frac{2\frac{d\sqrt{h}}{dx}\theta'+\sqrt{h}\theta''}{\sqrt{h}\theta'^2}W'\right)+
\left(\frac{d^2\sqrt{h}}{dx^2}+\lambda f'(u)\sqrt{h}+\mu \sqrt{h}\right)W\nonumber\\
&=\sqrt{h}\theta'^2\left(W''+\frac{1}{h\theta'^2}\frac{d}{dx}(h\theta')W'+\frac{\lambda\rho}{2h^2\theta'^2}W\right).\label{S1E7+1}
\end{align}
If $h(u(x))\theta'(x)$ is constant $C$, then
\[
\theta(x)=\theta_0+C\int_0^x\frac{dy}{h(u(y))}\ \ \textrm{for some}\ \theta_0.
\]
Moreover, it follows from (\ref{S1E7+1}) that $W$ satisfies a second order linear ODE with constant coefficients, and hence the solution $W$ has an exact expression.
If (\ref{S1E5_5}) has an exact solution $h(u)$, then $\theta$ can be written explicitly.
Because of (\ref{S1E6+1}), the eigenfunction $\varphi(x)$ has an exact expression.
What we have to do is to find an exact solution $h(u)$ of (\ref{S1E5_5}).
When $f(u)=e^u$, by direct calculation we see that (\ref{S1E5_5}) has the exact solution
\begin{equation}\label{S1E8}
h(u)=\frac{2\mu}{\lambda}+e^{\alpha}-e^u.
\end{equation}

\section{The case $\mu>0$}
In Sections 3, 4, and 5 we consider the case $f(u)=e^u$.
\begin{lemma}\label{L1}
Let $\{\mu_j\}_{j=1}^{\infty}$, $\mu_1<\mu_2<\cdots$, be the eigenvalues of (\ref{EVP+}), and let $\varphi_j(x)$ be the eigenfunction corresponding to $\mu_j$.\\
(i) For $j\ge 2$, $\mu_j>0$ and $\mu_j$ is the unique solution of the problem (\ref{T1E1}).
Moreover, $\varphi_j(x)$, $j\ge 2$, can be written as (\ref{T2E1}).\\
(ii) If $0<\tau<\tau_1$, then $\mu_1$ is the unique solution of the problem (\ref{T1E1}), and $\varphi_1(x)$ can be written as (\ref{T2E1}).
\end{lemma}
\begin{proof}
Let $f(u)=e^u$ and let $\tau>0$.
We consider the case $\mu>0$.
The function (\ref{S1E8}) is an exact solution of (\ref{S1E5_5}).
Let $a:=\sqrt{\mu}/\tau$.
Substituting (\ref{tau+}) and (\ref{S1E0}) into (\ref{S1E8}), we have
\begin{equation}\label{S1E7-3}
h(u(x))=\cosh^2(\tau)\left(a^2+\tanh^2(\tau x)\right).
\end{equation}
We see that $h(u(x))>0$ for $-1\le x\le 1$.

Now we determine $\theta$ such that
\begin{equation}\label{S1E7-2}
\frac{dh(u(x))\theta'(x)}{dx}=0\ \ \textrm{and}\ \ 
\theta'(0)=\sqrt{\frac{\lambda\rho}{2}}\frac{1}{h(u(0))}.
\end{equation}
Then,
\[
\theta(x)=\sqrt{\frac{\lambda\rho}{2}}\int_0^x\frac{dy}{h(u(y))}
\]
is one solution of (\ref{S1E7-2}).
Because of (\ref{S1E7+1}), $W$ satisfies $W''+W=0$.
Since $W(\theta)=\cos (\theta+\theta_1)$, it follows from (\ref{S1E6+1}) that the eigenfunction can be written as
\begin{equation}\label{S1E7-1}
\varphi(x)=\sqrt{h(u(x))}\cos\left(\sqrt{\frac{\lambda\rho}{2}}\int_0^x\frac{dy}{h(u(y))}+\theta_1\right).
\end{equation}
Substituting (\ref{S1E0}) and (\ref{S1E8}) into (\ref{S1E6+1-}), we have
\begin{equation}\label{S1E7-1+0}
\rho=a^2\left(a^2+1\right)^2\cosh^6\tau.
\end{equation}
Since $\mu>0$, we see that $\rho>0$.
By (\ref{S1E0}) and (\ref{S1E7-1+0}) we have
\begin{equation}\label{S1E7-1+1}
\sqrt{\frac{\lambda\rho}{2}}=a\left(a^2+1\right)\tau\cosh^2\tau.
\end{equation}
Using (\ref{S1E7-3}), (\ref{S1E7-1+1}), and the identity
\[
\int_0^x\frac{dy}{a^2+\tanh^2(\tau y)}=\frac{1}{a(a^2+1)\tau}\left(a\tau x+\arctan\left(\frac{\tanh(\tau x)}{a}\right)\right),
\]
we have
\begin{align}
\sqrt{\frac{\lambda\rho}{2}}\int_0^x\frac{dy}{h(u(y))}
&=\sqrt{\frac{\lambda\rho}{2}}\frac{1}{\cosh^2\tau}\int_0^x\frac{dy}{a^2+\tanh^2(\tau y)}\nonumber\\
&=\sqrt{\mu}x+\arctan\left(\frac{\tau\tanh(\tau x)}{\sqrt{\mu}}\right).\label{S1E7+0}
\end{align}
Substituting (\ref{S1E7-3}) and (\ref{S1E7+0}) into (\ref{S1E7-1}), we have
\[
\varphi(x)=\cosh(\tau)\sqrt{\frac{\mu}{\tau^2}+\tanh^2(\tau x)}
\cos\left(\sqrt{\mu}x+\arctan\left(\frac{\tau\tanh(\tau x)}{\sqrt{\mu}}\right)+\theta_1\right).
\]
Odd and even eigenfunctions can be written as
\begin{align*}
\varphi^{o}(x)&=\cosh(\tau)\sqrt{\frac{\mu}{\tau^2}+\tanh^2(\tau x)}
\sin\left(\sqrt{\mu}x+\arctan\left(\frac{\tau\tanh(\tau x)}{\sqrt{\mu}}\right)\right) \ \ \textrm{and}\\
\varphi^{e}(x)&=\cosh(\tau)\sqrt{\frac{\mu}{\tau^2}+\tanh^2(\tau x)}
\cos\left(\sqrt{\mu}x+\arctan\left(\frac{\tau\tanh(\tau x)}{\sqrt{\mu}}\right)\right),
\end{align*}
respectively.
First, we consider $\varphi_j(x)$ in the case where $j\ge 2$ is even, i.e., $j=2k$.
Then, it follows from Proposition~\ref{P1}~(iii) that the eigenfunction is odd, and hence, $\varphi^o(x)$ is a candidate.
Since the Dirichlet boundary condition is satisfied, $\varphi^o(\pm 1)=0$, i.e.,
\[
\sqrt{\mu}+\arctan\left(\frac{\tau\tanh\tau}{\sqrt{\mu}}\right)=\pi n\ \textrm{for some}\ n\ge 1.
\]
When $\mu$ satisfies the above equation, by the definition of $\varphi^o(x)$ we see that $\sharp\{x\in(-1,1)|\ \varphi^o(x)=0\}=2n-1$.
It follows from Proposition~\ref{P1}~(ii) that the zero number satisfies $2n-1=j-1$, i.e., $n=j/2$.
We have
\[
\sqrt{\mu}+\arctan\left(\frac{\tau\tanh\tau}{\sqrt{\mu}}\right)=\frac{\pi}{2}j.
\]
Therefore,
\[
\tan\left(\frac{\pi}{2}j-\sqrt{\mu}\right)=\frac{\tau\tanh\tau}{\sqrt{\mu}}\ \ \textrm{and}\ \ \frac{\pi}{2}(j-1)<\sqrt{\mu}<\frac{\pi}{2}j.
\]
This is equivalent to (\ref{T1E1}).
The equation (\ref{T1E1}) has a unique solution $\mu_j$ for $\tau>0$, which is the $j$-th eigenvalue.
Consequently, $\varphi_j(x)=\varphi^o(x)/\cosh(\tau)$ with $\mu=\mu_j$.

Second, we consider $\varphi_j(x)$ in the case where $j\ge 1$ is odd, i.e., $j=2k-1$.
By the same argument we obtain the same equation (\ref{T1E1}) even if $j$ is odd.
When $j\in\{3,5,7,\ldots\}$, we see that (\ref{T1E1}) has a unique solution $\mu_j$ for $\tau>0$, which is the $j$-th eigenvalue.
When $j=1$, we see that (\ref{T1E1}) has a unique solution $\mu_1$ provided that $0<\tau<\tau_1$.

When $j\in\{2,3,4,\ldots\}$, $\varphi^o(x)$ and $\varphi^e(x)$ can be summarized as (\ref{T2E1}) up to multiple constant and $\mu_j$ is given by (\ref{T1E1}).
The assertion (i) holds.
If $0<\tau<\tau_1$, then $\varphi_1(x)$ and $\mu_1$ are given by (\ref{T2E1}) and (\ref{T1E1}), respectively.
The assertion (ii) holds.
The proof is complete.
\end{proof}
In Lemma~\ref{L1} we have obtained $\{(\mu_j,\varphi_j(x))|\ j\ge 2\}$ for all $\tau>0$ and $(\mu_1,\varphi_1(x))$ for $0<\tau<\tau_1$.
In Sections 4 and 5 we study $(\mu_1,\varphi_1(x))$ for $\tau\ge\tau_1$.

\section{The case $\mu=0$}
We study $(\mu_1,\varphi_1(x))$ in the case $\tau=\tau_1$.
\begin{lemma}\label{L2}
Let $\mu_1$ be the first eigenvalue of (\ref{EVP+}), and let $\varphi_1(x)$ be the first eigenfunction.
If $\tau=\tau_1$, then $\mu_1=0$ and $\varphi_1(x)$ is given by (\ref{T2E2}).
\end{lemma}
\begin{proof}
Let $(\lambda(\tau),u(x,\tau))$ be a solution of (\ref{GP+}) given by (\ref{S1E0}).
Differentiating $u$ with respect to $\tau$, we have
\[
u_{\tau}(x,\tau)=2(\tanh\tau-x\tanh(\tau x)).
\]
Note that $u_{\tau}(\pm 1,\tau)=0$.
On the other hand, differentiating $u''+\lambda e^u=0$ with respect to $\tau$, we have
\[
u_{\tau}''+\lambda e^uu_{\tau}=-\lambda'e^u,
\]
where $\lambda'(\tau)$ is given by (\ref{S1E1+2}).
Then, $\tau_1$ is the unique positive solution of $\lambda'(\tau)=0$, where $\tau_1$ is a solution of (\ref{S1E1+1}).
When $\tau=\tau_1$, we see that $u_{\tau}(x,\tau)$ satisfies (\ref{EVP+}) with $\mu=0$.
We easily see that $u_{\tau}(x,\tau)>0$ for $-1<x<1$.
Hence, $0$ is the first eigenvalue and $u_{\tau}(x,\tau_1)/2$, which is (\ref{T2E2}), is the first eigenfunction.
The proof is complete.
\end{proof}

\section{The case $\mu<0$}
We study $(\mu_1,\varphi_1(x))$ in the case $\tau>\tau_1$.
\begin{lemma}\label{L3}
Let $\mu_1$ be the first eigenvalue of (\ref{EVP+}), and let $\varphi_1(x)$ be the first eigenfunction.
If $\tau>\tau_1$, then $\mu_1<0$, $\mu_1$ is the unique negative solution of (\ref{T1E2}), and $\varphi_1(x)$ is given by (\ref{T2E3}).
\end{lemma}
\begin{proof}
We consider the case $\mu<0$.
Let $\bmu:=-\mu(>0)$ and $\ba:=\sqrt{\bmu}/\tau(>0)$.
We see that
\begin{equation}\label{S3E1-2}
h(u):=\frac{2\bmu}{\lambda}-e^{\alpha}+e^u
\end{equation}
is an exact solution of (\ref{S1E5_5}).
Note that the solution (\ref{S3E1-2}) is not the same as (\ref{S1E8}).
By (\ref{S1E0}) and (\ref{S3E1-2}) we see that
\begin{equation}\label{S3E1-2+}
h(u(x))=\cosh^2(\tau)\left(\ba^2-\tanh^2(\tau x)\right)
\end{equation}
and $h(u(0))=2\bmu/\lambda>0$.
We define $\varphi(x):=\sqrt{h(u(x))}W(\theta(x))$.
Note that $W(\theta)$ and $\theta(x)$ are not the same functions as in Section 3 and that they are defined later.
Hereafter, we work on the interval $I:=(-{\rm artanh}(\ba)/\tau,{\rm artanh}(\ba)/\tau)$, because $h(u(x))>0$ for $x\in I$.
Here, $w={\rm artanh}(z)$ denotes the inverse function of $z=\tanh(w)$ and ${\rm artanh}(z):=\infty$ if $z\ge 1$.
Note that the interval $I$ may not include $[-1,1]$.
By (\ref{S1E6+1-}) we have
\begin{equation}\label{S3E1-1}
\rho=-\ba^2(-\ba^2+1)^2\cosh^6\tau.
\end{equation}
Let $\brho:=-\rho(\ge 0)$.
We define $\theta(x)$ such that
\begin{equation}\label{S3E1}
\frac{dh(u(x))\theta'(x)}{dx}=0\ \ \textrm{and}\ \ \theta'(0)=\sqrt{\frac{\lambda\brho}{2}}\frac{1}{h(u(0))}.
\end{equation}
We see that
\begin{equation}\label{S3E1+}
\theta(x)=\sqrt{\frac{\lambda\brho}{2}}\int_0^x\frac{dy}{h(u(y))}
\end{equation}
is one solution of (\ref{S3E1}).
By (\ref{S1E7+1}) we see that $W(\theta)$ satisfies
\begin{equation}\label{S3E2}
W''-W=0.
\end{equation}
We consider the case $\ba=1$.
Then it follows from (\ref{S3E1-1}) that $\rho=0$.
We see by (\ref{S1E7}) that the first eigenfunction should be $\sqrt{h(u(x))}$.
However, by (\ref{S3E1-2+}) we see that $\sqrt{h(u(\pm 1))}\neq 0$, and hence the Dirichlet boundary condition is not satisfied.
The case $\bar{a}=1$ does not occur.

We consider the case $\bar{a}\neq 1$.
Using (\ref{S3E1-2+}) and the identities
\begin{align*}
\sqrt{\frac{\lambda\brho}{2}}&=\frac{1}{\sqrt{2}}\frac{\sqrt{2}\tau}{\cosh\tau}\ba|-\ba^2+1|\cosh^3\tau\ \ \textrm{and}\\
\int_0^x\frac{dy}{\ba^2-\tanh^2(\tau y)}
&=\frac{1}{\ba(\ba^2-1)\tau}\left(\ba\tau x-\frac{1}{2}\log\frac{\ba+\tanh(\tau x)}{\ba-\tanh(\tau x)}\right),
\end{align*}
we have
\begin{align}
\sqrt{\frac{\lambda\brho}{2}}\int_0^x\frac{dy}{h(u(y))}
&=\sqrt{\frac{\lambda\brho}{2}}\frac{1}{\cosh^2\tau}\int_0^x\frac{dy}{\ba^2-\tanh^2(\tau y)}\nonumber\\
&=
\begin{cases}
\bar{a}\tau x-\frac{1}{2}\log\frac{\bar{a}+\tanh(\tau x)}{\bar{a}-\tanh(\tau x)} & \textrm{if}\ \bar{a}>1,\\
-\bar{a}\tau x+\frac{1}{2}\log\frac{\bar{a}+\tanh(\tau x)}{\bar{a}-\tanh(\tau x)} & \textrm{if}\ 0<\bar{a}<1.\\
\end{cases}\label{S3E2+}
\end{align}
Note that the above function is well-defined in $I$.
On the other hand, $W(\theta)=\cosh(\theta)$ is a solution of (\ref{S3E2}) which is even.
Hence, by (\ref{S3E1+}) and (\ref{S3E2+}) we have
\begin{equation}\label{S3E2++}
W(\theta(x))=\frac{1}{2}\left( e^{\sqrt{\bmu}x}\sqrt{\frac{\ba-\tanh(\tau x)}{\ba+\tanh(\tau x)}}+e^{-\sqrt{\bmu}x}\sqrt{\frac{\ba+\tanh(\tau x)}{\ba-\tanh(\tau x)}}\right)
\end{equation}
in both case $0<\ba<1$ and case $\ba>1$.
By (\ref{S3E1-2+}) and (\ref{S3E2++}) we have
\begin{align*}
\varphi(x)&=\sqrt{h(u(x))}W(\theta(x))\\
&=\cosh(\tau)\left(\ba\cosh(\sqrt{\bmu}x)-\sinh(\sqrt{\bmu}x)\tanh(\tau x)\right).
\end{align*}
By direct calculation we see that $\varphi(x)$ is defined for all $x\in [-1,1]$ and $\varphi''+\lambda e^u\varphi-\bmu\varphi=0$ for all $x\in(-1,1)$, although $I$ may not include $[-1,1]$.
Thus, $\varphi(x)$ becomes an eigenfunction provided that $\varphi(x)$ satisfies the Dirichlet boundary condition.
We study the problem $\varphi(\pm 1)=0$.
Then, $\sqrt{\bmu}\cosh(\sqrt{\bmu})-\tau\sinh(\sqrt{\bmu})\tanh(\tau)=0$.
We have
\begin{equation}\label{S3E3}
\tanh\sqrt{\bmu}=\frac{\sqrt{\bmu}}{\tau\tanh\tau}.
\end{equation}
We easily see that if $\tau>\tau_1$, then $\tau\tanh\tau>1$, and hence (\ref{S3E3}) has a unique positive solution $\bmu_1$.
Moreover, we easily see that $\varphi(x)>0$ for $-1<x<1$.
Let $\mu_1:=-\bmu_1$.
When $\tau>\tau_1$, $\mu_1$ is the first eigenvalue of (\ref{EVP+}).
Thus, $\mu_1$ is the unique negative solution of (\ref{T1E2}).
When $\mu=\mu_1$, the function $\tau\varphi(x)/\cosh\tau$, which is (\ref{T2E3}), is the first eigenfunction.
The proof is complete.
\end{proof}
In the above proof we can show that the case $0<\ba\le 1$ does not occur.
Thus, the interval $I$ includes $[-1,1]$, and hence (\ref{S3E2+}) and (\ref{S3E2++}) are well-defined in $[-1,1]$.\\

\noindent
{\sc Proof of Theorems~\ref{T1} and \ref{T2}.}
Theorems~\ref{T1}~(i) and \ref{T2}~(i) follow from Lemma~\ref{L1}.
Theorems~\ref{T1}~(ii) and \ref{T2}~(ii) follow from Lemmas~\ref{L1}~(i) and \ref{L2}.
Theorems~\ref{T1}~(iii) and \ref{T2}~(iii) follow from Lemmas~\ref{L1}~(i) and \ref{L3}.
\qed
\\

\noindent
{\sc Proof of Corollary~\ref{C1}.}
(i) By (\ref{S1E0}) we see that $\lambda(\tau)e^{u(\tau)}\to 0$ in $C([-1,1])$ as $\tau\downarrow 0$.
Hence, an eigenvalue of (\ref{EVP+}) converges to an eigenvalue of the problem
\[
\begin{cases}
\varphi''=-\mu\varphi, & -1<x<1,\\
\varphi(-1)=\varphi(1)=0.
\end{cases}
\]
Thus, (\ref{C1E1}) holds, and the assertion (i) holds.\\
(ii) When $j\ge 2$, by (\ref{T1E1}) we see that
\[
\frac{\pi}{2}(j-1)\le\sqrt{\mu_j}\le\frac{\pi}{2}(j-1)+\arctan\left(\frac{\pi j}{2\tau\tanh\tau}\right).
\]
Since $\arctan({\pi j}/{(2\tau\tanh\tau)})\to 0$ ($\tau\to\infty$), (\ref{C1E2}) holds for $j\ge 2$.
We consider the case $j=1$.
Let $y(x):=(\tanh x)/x$, $x>0$.
Then we easily see that $\lim_{x\downarrow 0}y(x)=1$, $\lim_{x\to\infty}y(x)=0$, and $y'(x)<0$ for $x>0$.
Thus, if $\tau>\tau_1$, then the solution of $y(x)=1/(\tau\tanh\tau)$ is unique and it diverges as $\tau\to\infty$.
Thus, we see by (\ref{T1E2}) that (\ref{C1E2}) holds for $j=1$.
The assertion (ii) holds.
\qed
\bigskip

\noindent
{\sc Proof of Corollary~\ref{C2}.}
(i) It follows from Corollary~\ref{C1}~(ii) that $\sqrt{-\mu_1}\to\infty$ as $\tau\to\infty$.
We see by (\ref{T1E2}) that
\[
\frac{\sqrt{-\mu_1}}{\tau}=\tanh(\sqrt{-\mu_1})\tanh(\tau)\to 1\ \ \textrm{as}\ \ \tau\to\infty.
\]
As $\tau\to\infty$,
\begin{align*}
\frac{1}{\tau}\varphi_1\left(\frac{y}{\tau}\right)
&=\frac{\sqrt{-\mu_1}}{\tau}\cosh\left(\frac{\sqrt{-\mu_1}}{\tau}y\right)-\sinh\left(\frac{\sqrt{-\mu_1}}{\tau}y\right)\tanh y\\
&\to\cosh y-\sinh y\tanh y=\frac{1}{\cosh y}\ \ \textrm{in}\ \ C_{loc}(\R).
\end{align*}
(ii) As $\tau\to\infty$, $\tanh (\tau x)\to {\rm sign}(x)$ in $C_{loc}([-1,0)\cup(0,1])$, where
\[
{\rm sign}(x):=
\begin{cases}
1 & x>0,\\
0 & x=0,\\
-1 & x<0.
\end{cases}
\]
Thus, the conclusion follows from (\ref{T2E1}) and (\ref{C1E2}).
\qed

\section{The problem $\Delta u+\lambda e^{-u}=0$}
In this section we study (\ref{GP-}) and the linearized eigenvalue problem (\ref{EVP-}).
The set of the positive solutions of (\ref{GP-}) becomes a curve and it can be written explicitly as
\[
(\lambda,u)=\left(2e^{\alpha}\arctan^2\sqrt{e^{\alpha}-1},\alpha-\log\left(\tan^2\left(\left(\arctan\sqrt{e^{\alpha}-1}\right)x\right)+1\right)\right),
\]
where $\alpha=\left\|u\right\|_{L^{\infty}([-1,1])}$.
Let $\tau(\alpha):=\arctan\sqrt{e^{\alpha}-1}$.
Then,
\begin{equation}\label{tau-}
\alpha=-2\log\cos\tau,
\end{equation}
$\tau(0)=0$, $\tau(\infty)=\pi/2$, and $\tau$ is a homeomorphism from $[0,\infty)$ to $[0,\pi/2)$.
We use $\tau$ ($0\le\tau<\pi/2$) as an independent variable instead of $\alpha$.
We have
\begin{equation}\label{S6E1}
(\lambda(\tau),u(x,\tau))=\left(\frac{2\tau^2}{\cos^2\tau},2\log\frac{\cos(\tau x)}{\cos\tau}\right).
\end{equation}
The curve has no turning point, since $\lambda'(\tau)=4\tau(1+\tau\tan\tau)/(\cos^2\tau)>0$ for $0<\tau<\pi/2$.
We consider (\ref{EVP-}) which is the linearized eigenvalue problem of (\ref{GP-}) at the solution (\ref{S6E1}).
\begin{theorem}\label{T3}
Let $0<\tau<\pi/2$, and let $\{\mu_j\}_{j=1}^{\infty}$, $\mu_1<\mu_2<\cdots$, be the eigenvalues of (\ref{EVP-}).
Then the following hold:\\
(i) For $j\ge 1$, $\mu_j>0$ and $\mu_j$ is given by the unique solution of the problem
\begin{equation}\label{T3E1}
\tan\left(\frac{\pi}{2}(j+1)-\sqrt{\mu_j}\right)=\frac{\sqrt{\mu_j}}{\tau\tan\tau}\ \ \textrm{and}\ \ \frac{\pi}{2}j<\sqrt{\mu_j}<\frac{\pi}{2}(j+1).
\end{equation}
(ii) The eigenfunction $\varphi_j(x)$ corresponding to $\mu_j$ is given by
\begin{equation}\label{T3E2}
\varphi_j(x)=\sqrt{\frac{\mu_j}{\tau^2}+\tan^2(\tau x)}\sin\left(\sqrt{\mu_j}x-\arctan\left(\frac{\tau\tan(\tau x)}{\sqrt{\mu_j}}\right)+\frac{\pi}{2}j\right).
\end{equation}
(iii) For $j\ge 1$,
\[
\lim_{\tau\downarrow 0}\sqrt{\mu_j}=\frac{\pi}{2}j\ \ \textrm{and}\ \ \lim_{\tau\uparrow\pi/2}\sqrt{\mu_j}=\frac{\pi}{2}(j+1).
\]
(iv) Let $\varphi_j(x)$ be the eigenfunction given by (\ref{T3E2}).
For $j\ge 1$, as $\tau\uparrow\pi/2$,
\[
\varphi_j(x)\to\bar{\varphi}_j(x)\ \ \textrm{in}\ \ C_{loc}((-1,1)),
\]
where
\[
\bar{\varphi}_j(x)=\sqrt{(j+1)^2+\tan^2\left(\frac{\pi}{2}x\right)}
\sin\left(\frac{\pi}{2}(j+1)x-\arctan\left(\frac{\tan\left(\frac{\pi}{2}x\right)}{j+1}\right)+\frac{\pi}{2}j\right).
\]
\begin{proof}
We briefly prove Theorem~\ref{T3}.
The proof is almost the same as the case (\ref{GP+}).

Let $f(u)=e^{-u}$ and let $0<\tau<\pi/2$.
We consider the case $\mu>0$.
The function
\begin{equation}\label{T3PE1}
h(u)=\frac{2\mu}{\lambda}-e^{-\alpha}+e^{-u}
\end{equation}
is an exact solution of (\ref{S1E5_5}).
Let $a:=\sqrt{\mu}/\tau$.
Substituting (\ref{tau-}) and (\ref{S6E1}) into (\ref{T3PE1}), we have
\[
h(u(x))=\cos^2\tau\left(a^2+\tan^2(\tau x)\right).
\]
We see that $h(u(x))>0$ for $-1\le x\le 1$.
Substituting (\ref{tau-}) and (\ref{T3PE1}) into (\ref{S1E6+1-}), we have
\begin{equation}\label{T3PE3}
\rho=a^2(a^2-1)^2\cos^6\tau.
\end{equation}
We consider the case $a=1$.
Then, it follows from (\ref{T3PE3}) that $\rho=0$.
Hence, by (\ref{S1E7}) we see that the eigenfunction should be $\sqrt{h(u(x))}$.
However, $\sqrt{h(u(\pm 1))}\neq 0$, and hence the Dirichlet boundary condition is not satisfied.
Thus, the case $a=1$ does not occur.
Since $\sqrt{\mu_1}$ depends continuously on $\tau$ and $a\neq 1$, we see that $\sqrt{\mu_1}<\tau$ for $0\le\tau<\pi/2$ or $\sqrt{\mu_1}>\tau$ for $0\le\tau<\pi/2$.
When $\tau=0$, it is clear that $\sqrt{\mu_1}=\pi/2>0=\tau$, and hence $\sqrt{\mu_1}>\tau$ for $0\le\tau<\pi/2$.
Since $\sqrt{\mu_j}\ge\sqrt{\mu_1}(>\tau)$ for $j\ge 1$, the case $a<1$ does not occur.

It is enough to consider the case $a>1$.
Since $\rho>0$, the candidate of the eigenfunction $\varphi(x)$ is given by (\ref{S1E7-1}), where $\sqrt{h(u(x))}=\cos\tau\sqrt{a^2+\tan^2(\tau x)}$
and
\begin{align*}
\sqrt{\frac{\lambda\rho}{2}}\int_0^x\frac{dy}{h(u(y))}
&=\sqrt{\frac{\lambda\rho}{2}}\frac{1}{\cos^2\tau}\int_0^x\frac{dy}{a^2+\tan^2(\tau y)}\\
&=\sqrt{\frac{\lambda\rho}{2}}\frac{1}{a(a^2-1)\cos^2\tau}\left(a\tau x-\arctan\left(\frac{\tan(\tau x)}{a}\right)\right)\\
&=a\tau x-\arctan\left(\frac{\tan(\tau x)}{a}\right).
\end{align*}
Thus,
\[
\varphi(x)=\cos\tau\sqrt{\frac{\mu}{\tau^2}+\tan^2(\tau x)}\cos\left(\sqrt{\mu}x-\arctan\left(\frac{\tau\tan(\tau x)}{\sqrt{\mu}}\right)+\theta_1\right).
\]
By the same argument as in the proof of Lemma~\ref{L1} we see that, for $j\ge 1$, $\sqrt{\mu_j}$ satisfies
\begin{equation}\label{T3PE4}
\sqrt{\mu_j}-\arctan\left(\frac{\tau\tan\tau}{\sqrt{\mu_j}}\right)=\frac{\pi}{2}j\ \ \textrm{and}\ \ 
\frac{\pi}{2}j<\sqrt{\mu_j}<\frac{\pi}{2}(j+1).
\end{equation}
Here, we use the fact that the function $\sqrt{\mu_j}x-\arctan(\tau\tan(\tau x)/\sqrt{\mu_j})$ is monotonically increasing for $-1<x<1$ provided that $0<\tau<\pi/2$.
The problem (\ref{T3PE4}) is equivalent to (\ref{T3E1}).
In particular, $\mu_j>0$ for $j\ge 1$.
We need not consider the case $\mu\le 0$.
The proof of (i) is complete.
By the same way as in Lemma~\ref{L1} we see that (\ref{T3E2}) is an eigenfunction corresponding to $\mu_j$,
The proof of (ii) is complete.

The assertions (iii) and (iv) clearly follow from (\ref{T3E1}) and (\ref{T3E2}), respectively.
All the proofs are complete.
\end{proof}

\end{theorem}

\end{document}